\documentclass[12pt]{amsart}
\usepackage[foot]{amsaddr}
\usepackage{amsmath, amssymb, amsthm, thmtools, hyperref}
\usepackage[sort]{cite}
\usepackage[charter]{mathdesign}
\usepackage[margin=1.25in]{geometry}
\usepackage{eucal}
\usepackage{tikz-cd}

\title{The canonical syzygy conjecture for ribbons}
\author{Anand Deopurkar}
\address{The University of Georgia}
\address{Email: \url{deopurkar@uga.edu}}

\DeclareMathOperator{\Cliff}{Cliff}
\DeclareMathOperator{\id}{id}
\DeclareMathOperator{\im}{im}

\newcommand{\stack}[1]{\mathcal{#1}}
\newcommand{\orb}[1]{{#1}}

\newcommand{\red}{\rm red}
\renewcommand{\O}{{\mathcal O}}
\renewcommand{\k}{\mathbb{K}}

\ifcsname theorem\endcsname{}\else\declaretheorem[parent=section]{theorem}\fi
\ifcsname corollary\endcsname{}\else\declaretheorem[sibling=theorem]{corollary}\fi
\ifcsname lemma\endcsname{}\else\declaretheorem[sibling=theorem]{lemma}\fi
\ifcsname proposition\endcsname{}\else\declaretheorem[sibling=theorem]{proposition}\fi
\ifcsname conjecture\endcsname{}\else\declaretheorem[sibling=theorem]{conjecture}\fi
\ifcsname problem\endcsname{}\else\fi
\ifcsname question\endcsname{}\else\fi
\ifcsname definition\endcsname{}\else\fi
\ifcsname exercise\endcsname{}\else\fi
\ifcsname example\endcsname{}\else\fi
\ifcsname remark\endcsname{}\fi

\renewcommand {\P}{{\bf P}}

\providecommand {\A}{{\bf A}}

\providecommand {\from}{{\colon}}

\providecommand{\spec}{\operatorname{Spec}}

\providecommand{\coker}{\operatorname{coker}}

\providecommand{\Hom}{\operatorname{Hom}}
\providecommand{\Ext}{\operatorname{Ext}}

\begin{document}

\maketitle

\begin{abstract}
  Green's canonical syzygy conjecture asserts a simple relationship between the Clifford index of a smooth projective curve and the shape of the minimal free resolution of its homogeneous ideal in the canonical embedding.
  We prove the analogue of this conjecture formulated by Bayer and Eisenbud for a class of non-reduced curves called ribbons.
  Our proof uses the results of Voisin and Hirschowitz--Ramanan on Green's conjecture for general smooth curves.
\end{abstract}

\section{Introduction}
Let $X$ be a projective variety over a field $\k$ and $L$ a very ample line bundle on $X$.
Consider the embedding $X \subset \P^n$ given by the complete linear system $|L|$.
Associated to $X \subset \P^n$ is the homogeneous ideal $I \subset S = \k[x_0, \dots, x_n]$.
The data of $(X, L)$ and the data of $(S, I)$ represent two different points of view of studying the same object---a geometer would rather work with $(X, L)$ whereas an algebraist would rather work with $(S, I)$.
But since $(X, L)$ and $(S, I)$ contain the same information, there ought to be a connection between the properties of $(X, L)$ that a geometer would study with the properties of $(S, I)$ that an algebraist would.

The question of finding relationships between the geometric properties of $(X, L)$ and the algebraic properties of $(S, I)$ has led to fascinating theorems and conjectures.
One of the most deeply studied cases is where $X$ is a smooth curve (but see \cite{gre:84} and \cite{ein.laz:12} for more general results, especially in an asymptotic setting).
In this case, the work of Green \cite{gre:84} and Green--Lazarsfeld \cite{gre.laz:86} predicts a simple relationship between the Clifford index of $X$ and the shape of the minimal free resolution of $S/I$ as an $S$ module.
When $L$ is the canonical bundle, this relationship is the content of Green's canonical syzygy conjecture, which we now recall.
We state it in terms of the Koszul cohomology groups $K_{p,q}$.
Recall that when $X \subset \P^n$ is projectively normal, $K_{p,q}(X, L)$ is naturally identified with the graded component of degree $(p+q)$ in the $p$th term of the minimal free resolution of $S/I$.

From here on, let $\k$ be an algebraically closed field of characteristic zero.
\begin{conjecture}[Green's canonical syzygy conjecture \cite{gre:84}]
  \label{conj:green}
  Let $C$ be a smooth projective curve of genus $g \geq 2$ over $\k$.
  The Koszul cohomology group $K_{p,2}(C, \omega_C)$ vanishes if and only if $p$ is smaller than the Clifford index of $C$.
\end{conjecture}
The statement of the conjecture generalizes classical results of Noether and Petri, namely that a canonically embedded curve is projectively normal, and that its homogeneous ideal is generated by quadrics if it does not carry a $g^1_3$ or a $g^2_5$.

In 1995, Bayer and Eisenbud laid out an approach to prove a generic version of Green's conjecture by a degeneration argument.
They considered a special class of degenerate curves, called \emph{ribbons}.
These are double structures on $\P^1$ locally isomorphic to $\spec \k[s, \epsilon]/\epsilon^2$, and they arise as flat limits of families of canonically embedded curves as the curves approach a hyperelliptic curve.
Despite being non-reduced, ribbons retain many nice features of smooth curves, such as a projectively normal canonical embedding and a sensible notion of Clifford index.
In terms of this Clifford index, we have the following analogue of \autoref{conj:green} due to Bayer and Eisenbud.
\begin{conjecture}[Canonical syzygy conjecture for ribbons \cite{bay.eis:95}]
  \label{conj:ribbon-green}
  Let $C$ be a ribbon over $\k$ of arithmetic genus $g$.
  Then $K_{p,2}(C, \omega_C)$ vanishes if and only if $p$ is smaller than the (ribbon) Clifford index of $C$.
\end{conjecture}
By the semi-continuity of the Clifford index and the smoothability results of Fong \cite{fon:93}, \autoref{conj:ribbon-green} implies \autoref{conj:green} for a general curve of every Clifford index (see the proof of \autoref{cor:generic_green}).

Having described the conjectures, let us review what is known.
Schreyer and Teixidor I Bigas settled \autoref{conj:green} for general $p$-gonal curves of genus $g$, where $g$ is large compared to $p$.
Schreyer settled the cases $g > (p-1)(p-2)$ \cite{sch:89}; Teixidor I Bigas advanced it much further to handle $g \geq 3p+2$ \cite{tei-i-big:02}.
In two breakthrough papers, Voisin proved \autoref{conj:green} when $C$ is general in moduli \cite{voi:02,voi:05}.
Combined with prior work of Hirschowitz and Ramanan \cite{hir.ram:98}, her work implies that \autoref{conj:green} holds for every smooth curve of odd genus and maximum Clifford index.
Since then, several authors---Aprodu, Farkas, Lelli-Chiesa, and Pacienza---have made remarkable progress on \autoref{conj:green}; it has been proved for general curves of every gonality, for curves lying on a K3 surface, and for certain curves on rational surfaces \cite{apr:05,apr.pac:08,apr.far:11*1,apr.far:12,lel:13}.
Yet, the full conjecture remains unproved.

In this paper, we prove \autoref{conj:ribbon-green}.
Our proof uses the results of Voisin and Hirschowitz--Ramanan.
 \begin{theorem}
  \label{thm:main}
  The canonical syzygy conjecture (\autoref{conj:ribbon-green}) holds for every ribbon.
\end{theorem}
As a result, we get another proof of the following.
\begin{corollary}\label{cor:generic_green}
  Green's canonical syzygy conjecture (\autoref{conj:green}) holds for a non-empty open subset of curves of a given genus and Clifford index.
\end{corollary}
\begin{proof}
  In the appendix to \cite{gre:84}, Green and Lazarsfeld show that $K_{p,2}(C) \neq 0$ for $p \geq c$ for \emph{every} smooth curve $C$ of Clifford index $c$.
  The hard part is the converse.

  We now show the converse for a non-empty open subset of curves with Clifford index $c$.
  Start with a ribbon $C_0$ of Clifford index $c$.
  By \autoref{thm:main}, we have $K_{p,2}(C_0, \omega_{C_0}) = 0$ for $p < c$.
  By \cite[Theorem~2]{fon:93}, $C_0$ is the limit of a family of curves $C_t$ whose generic member is a smooth curve of Clifford index $c$.
  Since the vanishing of $K_{p,2}(C_t, \omega_{C_t})$ is an open condition, and it holds for $t = 0$, it holds for the generic $t$, and hence for a non-empty open subset of curves with Clifford index $c$.
\end{proof}
The proof of Theorem~2 of \cite{fon:93} shows that a ribbon of Clifford index $c$ is a limit of a family of curves whose generic member is a smooth curve of gonality $c+2$.
Therefore, the argument of \autoref{cor:generic_green} in fact yields the following stronger statement.
\begin{corollary}
  Green's conjecture (\autoref{conj:green}) holds for a dense open subset of curves of a given genus and gonality.
\end{corollary}

It would be wonderful to have a proof of \autoref{conj:ribbon-green} independent of Voisin's proof of \autoref{conj:green} for general curves.
It would give a new proof of Voisin's difficult theorem and would address the original motivation of Bayer and Eisenbud behind studying ribbons.
Nevertheless, any proof is better than no proof.

Irrespective of the original motivation, the canonical syzygy conjecture for ribbons is important for another reason, which comes from recent progress in the log minimal model program for $\overline{M}_g$.
For a canonical curve $C \subset \P^{g-1}$ with vanishing $K_{p,2}(C, \omega_C)$, we can define a point in a Grassmannian called the \emph{$p$th syzygy point} of $C$.
This point encodes the vector space of $p$th syzygies among the generators of the homogeneous ideal of $C$ (see \cite{deo.fed.swi:16} for details).
The GIT quotients of the loci of $p$th syzygy points are expected to lead to the canonical model of $\overline M_{g}$.
To describe this canonical model, it is essential to know which curves (if any) have GIT semi-stable syzygy points.
We expect that the $p$th syzygy point of a general ribbon will be GIT semi-stable.
The vanishing of $K_{p,2}$ shows that the $p$th syzygy point is at least well-defined.

The paper is organized as follows.
\autoref{sec:ribbons} contains basic results about ribbons and their deformations.
\autoref{sec:generic} proves \autoref{thm:main} for ribbons of odd genus and maximum Clifford index.
\autoref{sec:all} extends the result to all ribbons.
All schemes and stacks are locally of finite type over $\k$.

\section{Ribbons}\label{sec:ribbons}
We quickly review the theory of ribbons from \cite{bay.eis:95}.
Let $D$ be a reduced and connected scheme over $\k$.
A \emph{ribbon} over $D$ is a scheme $C$ with an isomorphism $D \to C_{\red}$ such that the ideal $I$ of $D \subset C$ satisfies $I^2 = 0$ and is locally free of rank 1 when considered as a sheaf on $D$.
A ribbon over $\P^1$ is called a \emph{rational ribbon}.
All our ribbons will be rational, so we drop this adjective.

A ribbon $C$ of arithmetic genus $g$ gives an exact sequence
\begin{equation}\label{eqn:cotangent}
  0 \to \O_{\P^1}(-g-1) \to \Omega_C|_{\P^1} \to \Omega_{\P^1} \to 0.
\end{equation}
We have an isomorphism $\Omega_C|_{\P^1} \cong \O_{\P^1}(-a-2) \oplus \O_{\P^1}(-b-2)$ for some integers $a$ and $b$ with $0 \leq a \leq b \leq g-1$ and $a+b = g-1$.
The \emph{Clifford index $\Cliff(C)$} of $C$ is defined to be the integer $a$.
We say that $C$ is \emph{hyperelliptic} if the following equivalent conditions hold: (i) the inclusion $\P^1 \subset C$ admits a retraction $C \to \P^1$, (ii) $\Cliff(C) = 0$, (iii) the sequence in \eqref{eqn:cotangent} is split.

A ribbon of arithmetic genus $g$ may be described explicitly as obtained from gluing $U_1 = \spec \k[s,\epsilon]/\epsilon^2$ and $U_2 = \spec \k[t, \eta]/\eta^2$ by isomorphisms
\begin{equation}
  \label{eqn:gluing}
  \begin{split}
    \epsilon = t^{-g-1} \eta \\
    s^{-1} = t + F(t) \eta,
  \end{split}
\end{equation}
on $U_1 \cap U_2$, where $F(t) \in \k[t, t^{-1}]$.
In fact, what matters is only the image of $F(t)$ in $\k[t,t^{-1}]/(\k[t]+t^{-g+1}\k[t^{-1}])$.
We abuse notation and denote this image also by $F(t)$.
The element $F(t)$ corresponds to the extension \eqref{eqn:cotangent} in $\Ext^1(\Omega_{\P^1}, \O_{\P^1}(-g-1))$.
In particular, $F(t) = 0$ if and only if $C$ is hyperelliptic.
Two ribbons $C_1$ and $C_2$ are isomorphic by a map that restricts to the identity on the underlying $\P^1$ if and only if $F_1(t) = cF_2(t)$ for some $c \in \k^*$.

Let $x \in C$ be a closed point, $\beta \from C' \to C$ the blow up at $x$, and $E \subset C'$ the exceptional divisor.
Then $C'$ is a ribbon of arithmetic genus $g-1$ and $\omega_{C'} = \beta^*\omega_C(-E)$.
Blowing up is related to the Clifford index as follows.
\begin{proposition}[See {\cite[Corollary~2.5]{bay.eis:95}}]
  \label{thm:blow-up-clifford}
The Clifford index  $\Cliff(C)$ is the smallest number $k$ such that there exists a sequence
  \[ C_k \to \dots \to C_1 \to C_0 = C\]
  where $C_{i+1} \to C_i$ is a blow-up at a closed point and $C_k$ is hyperelliptic.
\end{proposition}
The usual Clifford index of smooth curves and the Clifford index for ribbons obey semi-continuity.
\begin{theorem}[See {\cite[Theorem~2.1]{eis.gre:95}}]
  \label{thm:sc}
  Let $C \to \Delta$ be a proper flat family over a DVR $\Delta$ where the geometric general fiber $C_{\overline \eta}$ is smooth and the special fiber $C_0$ is a ribbon.
  Then $\Cliff(C_0) \leq \Cliff(C_{\overline \eta})$.
\end{theorem}
All ribbons of a given genus and Clifford index are related.
\begin{theorem}[See {\cite[\S~8]{bay.eis:95}}]
  \label{thm:carpet}
  Let $c \geq 1$.
  There exists a surface $X_{g,c} \subset \P^g$ such that all canonically embedded ribbons of  genus $g$ and Clifford index $c$ are hyperplane sections of $X_{g,c}$.
  In particular, the graded betti numbers $\dim K_{p,q}(C, \omega_C)$ depend only on the genus $g$ and the Clifford index $c$ of $C$.
\end{theorem}

We need basic results about the deformation theory of ribbons.
\begin{proposition}
  \label{thm:smooth}
  A ribbon $C$ has a smooth versal deformation space.
\end{proposition}
\begin{proof}
  For non-hyperelliptic ribbons, this follows from \cite[Theorem~6.1]{bay.eis:95}.
  Here is another proof that works for all ribbons, using deformation theory.
  Denote by $T^i$ the deformation-obstruction functors of Lichtenbaum and Schlessinger for $i = 0, 1, 2$ (see \cite[\S~1.3]{har:10}).
  Let $T^i_C$ be the sheaf $T^i(C/\k, \O_C)$.
  Suppose $A \to A'$ is a surjection of local Artin $\k$-algebras with kernel $\k$.
  Let $C_{A'} \to \spec A'$ be a deformation of $C$.
  The obstructions to lifting $C_{A'} \to \spec A'$ to a deformation $C_A \to \spec A$ lie successively in $H^0(C, T^2_C)$, $H^1(C, T^1_C)$, and $H^2(C, T^0_C)$.
  Since $C$ has dimension 1, we have $H^2(C, T^0_C) = 0$.
  Since $C$ is a local complete intersection, we have $T^2_C = 0$.
  To compute $T^1_C$, we do a simple local computation.
  Consider $U = \spec \k[s, \epsilon]/\epsilon^2$.
  Set $S = \k[s, \epsilon]$, $J = \epsilon^2S$, and $R = S/J$.
  The cotangent complex for $U/\k$ is given by
  \[ L_{\bullet}: 0 \to J/J^2 \xrightarrow{d} \Omega_{S/\k} \otimes_S R.\]
  By definition, $T^1_{U} = H^1(\Hom_{R}(L_{\bullet}, R))$.
  Let $e$ be a generator of the free $R$ module $J/J^2$.
  Note that $\Omega_{S/\k} \otimes_SR  = R\langle ds, d\epsilon\rangle$ and $d(e) = \epsilon d \epsilon$.
  Denote by $\partial_s$, $\partial_\epsilon$, and $e^*$ the dual generators of $ds$, $d\epsilon$, and $e$, respectively.
  Then
  \[ T^1_{U} = \coker(d^* \from R \langle \partial_s, \partial_\epsilon \rangle \to R \langle e^* \rangle),\]
  where $d^* (\partial_s) = 0$ and $d^*(\partial_\epsilon) = \epsilon e^*$.
  Therefore, we get $T^1_U = R/\epsilon \langle e^* \rangle = \k[s]\langle e^* \rangle$.
  By a similar computation on $V = \spec \k[t,\eta]/\eta^2$ and gluing, we get $T^1_C \cong \O_{\P^1}(2g+2)$.
  Therefore, $H^1(C, T^1_C) = 0$ and hence deformations of $C$ are unobstructed.  
\end{proof}

\begin{proposition}\label{thm:codim2}
  Let $(U, 0)$ be a versal deformation space of $C$ and $C_U \to U$ a versal family.
  Let $N \subset U$ be the closed subset consisting of $u \in U$ where $C_u$ has worse than nodal singularities.
  Then the codimension of $N$ in $U$ at $0$ is at least 2.
\end{proposition}
Said differently, $N \subset U$ is the complement of the open set of $u \in U$ over which $C_u$ is a semi-stable curve.
\begin{proof}
  Note that all ribbons degenerate isotrivially to the hyperelliptic ribbon.
  Indeed, in the gluing description \eqref{eqn:gluing}, we may replace $F(t)$ by $cF(t)$ and take $c$ to $0$.
  Therefore, it suffices to prove the proposition for the hyperelliptic ribbon.
  Furthermore, since $(U, 0)$ is irreducible, it suffices to exhibit an irreducible pointed scheme $(T,0)$ and a map $\phi \from (T,0) \to (U,0)$ such that the codimension of $\phi^{-1}(N)$ in $T$ at $0$ is at least 2.
  Said differently, it suffices to construct a family of curves $C_T \to T$ such that $C_0$ is a hyperelliptic ribbon and the locus of $t \in T$ such that $C_t$ has worse than nodal singularities has codimension at least 2.
  Take $T = \A^{2g+3} = \A \langle a_0, \dots, a_{2g+2} \rangle$ with $0 = (0,\dots, 0)$ and let $C_T \to T$ be the family of hyperelliptic curves defined affine locally by
  \[ y^2 = a_{2g+2} x^{2g+2} + \dots + a_1 x + a_0.\]
  The locus of worse than nodal curves corresponds to the set of $(a_0, \dots, a_{2g+2})$ for which the polynomial $a_{2g+2} x^{2g+2} + \dots + a_1 x + a_0$ has a zero of multiplicity at least 3.
  It is easy to see that this locus has codimension 2.
\end{proof}

\section{Proof for ribbons of odd genus and maximum Clifford index}\label{sec:generic}
Let $X$ be a projective scheme and $L$ an invertible sheaf on $X$.
Set $V = H^0(X, L)$.
Recall that the Koszul cohomology group $K_{p,q}(X, L)$ is defined as the middle cohomology group in the sequence
\[\wedge^{p+1}V \otimes H^0(X, L^{q-1}) \to \wedge^p V \otimes H^0(X, L^q) \to \wedge^{p-1} V \otimes H^0(X, L^{q+1}),\]
where the map is given by
\[ v_1 \wedge \dots \wedge v_p \otimes w \mapsto \sum_{i=1}^p (-1)^i v_1 \wedge \dots \wedge \widehat{v_i} \wedge \dots \wedge v_p \otimes v_i w.\]

Let $g = 2k+1$ with $k \geq 1$.
Let $\stack U$ be the (non-separated) stack of all curves $C$ that satisfy the following conditions:
\begin{enumerate}
\item $C$ is Gorenstein of arithmetic genus $g$ and $h^0(C, \O_C) = 1$.
\item $\omega_C$ is very ample and embeds $C$ as an arithmetically Gorenstein subscheme in $\P^{g-1}$,
\item a versal deformation space of $C$ is smooth.
\end{enumerate}
All three are open conditions and hence $\stack U$ is an open substack of the stack of all curves constructed, for example, in \cite{hal:10}.
The first two conditions are more important than the third; the third is there just to avoid any problems about divisor theory.
Denote by $\stack U^{\rm nodal} \subset \stack U$ the open substack parametrizing curves with at worst nodes as singularities.
Then $\stack U^{\rm nodal} \subset \overline{\orb M}_g$.
In fact, $\stack U^{\rm nodal}$ is precisely $\overline{\orb M}_g^{\rm va}$, the stack of stable curves with a very ample dualizing sheaf; it is studied in \cite{apr:05}.
Note that ${\orb M}_g \cap \overline{\orb M_g}^{\rm va}$ is the complement in ${\orb M}_g$ of the hyperelliptic locus; we denote it by ${\orb M}_g^{\rm nh}$.
Also note that all non-hyperelliptic ribbons satisfy the three conditions---the first is clear; the second is \cite[Theorem~5.3]{bay.eis:95}; and the third is \autoref{thm:smooth}.

Let $\pi \from {\stack C} \to {\stack U}$ be the universal curve.
Set ${\stack V} = \pi_*\left(\omega_\pi\right)$.
Then $\stack V$ is locally free of rank $g$ on $\stack U$.
Consider the Koszul complex
\begin{equation*}\label{eqn:koszul}
  \begin{split}
    K_\bullet: \wedge^{k+1}{\stack V} \xrightarrow{\delta_1} \wedge^{k}{\stack V} \otimes {\stack V} \xrightarrow{\delta_2} \wedge^{k-1}{\stack V} \otimes \pi_*\left(\omega_\pi^2\right) \xrightarrow{\delta_3} \wedge^{k-2}{\stack V} \otimes \pi_*\left(\omega_\pi^3\right) \xrightarrow{\delta_4} \dots \\
    \dots \to
    \wedge^{k+1-p}{\stack V} \otimes \pi_*\left(\omega_\pi^{p}\right) \to
      \dots \to
      \wedge^{0}{\stack V} \otimes \pi_*\left(\omega_\pi^{k+1}\right).
\end{split}
\end{equation*}
Define $\stack E$ and $\stack F$ by 
\begin{align*}
  \stack E &= \coker \delta_1 \\
  \stack F &= \ker \delta_3.
\end{align*}
\begin{proposition}
  $\stack E$ and $\stack F$ are locally free on $\stack U$ of the same rank.
\end{proposition}
\begin{proof}
  $\stack E$ is clearly locally free.
  Consider a point $[C] \in \stack U$.
  Since the homogeneous coordinate ring of the canonical embedding of $C$ is Gorenstein, we have the duality
  \[K_{p,q}(C, \omega_C) \cong K_{g-2-p,3-q}(C, \omega_C)^*.\]
  In particular, we have $K_{p,q}(C, \omega_C) = 0$ for $p \geq 4$ and for $p = 3$ and $q < g-2$.
  Therefore, the complex $K_\bullet$ is exact from the third place onward: $\ker({\delta_{q+1}})=\im(\delta_q)$ for all $q \geq 3$.
  Since $K_\bullet$ is a finite complex of locally free sheaves, it follows that $\stack F = \ker \delta_3$ is locally free.

  It is easy to see that the rank of both $\stack E$ and $\stack F$ is $\frac{2 \cdot g!}{(k-1)!(k+1)!}$.
\end{proof}

The map $\delta_2$ in $K_\bullet$ induces a map $\delta \from \stack E \to \stack F$.
Observe that
\begin{equation}\label{eqn:D}
  \coker \delta|_{[C]} = K_{k-1,2}(C, \omega_C).
\end{equation}

Let $D \subset \stack U$ be defined by the vanishing of $\det(\delta)$.

Let $D_{k+1} \subset \stack U$ be the closure of the locus of $(k+1)$-gonal curves in ${\orb M}^{\rm nh}_g$.

\begin{proposition}\label{prop:equality_of_divisors}
  Let $C$ be a non-hyperelliptic ribbon of odd genus $2k+1$.
  We have the equality $D_{k+1} = D$ in an open subset containing $[C] \in \stack U$.
\end{proposition}
\begin{proof}
  The results of Voisin \cite{voi:05} and Hirschowitz--Ramanan \cite{hir.ram:98} give $D_{k+1} = D$ on ${\orb M}^{\rm nh}_g$.
  Consider the open inclusion ${\orb M}^{\rm nh}_g \subset \stack U^{\rm nodal}$.
  Note that $\stack U^{\rm nodal} \setminus {\orb M}^{\rm nh}_g$ has one divisorial component $\Delta_0$ whose general fiber corresponds to an irreducible nodal curve.
  In \cite{voi:05}, Voisin shows that $K_{k-1,2}(C, \omega_C) = 0$ for every $C$ in a linear series on a K3 surface, which includes irreducible nodal curves.
  Therefore $D$ does not contain $\Delta_0$ as a component.
  Since $D_{k+1}$ is the closure of a divisor on ${\orb M}^{\rm nh}_g$, it does not contain $\Delta_0$ as a component either.
  Therefore, the equality of divisors $D = D_{k+1}$ holds in codimension 1 on ${\stack U}^{\rm nodal}$ and hence on all of ${\stack U}^{\rm nodal}$.
  The same reasoning and \autoref{thm:codim2} implies that $D = D_{k+1}$ holds around every point in $\stack U$ corresponding to a ribbon.
\end{proof}

\begin{corollary}\label{thm:generic}
  For a ribbon $C$ of genus $2k+1$ with the maximum Clifford index $k$, we have $K_{k-1,2}(C, \omega_C) = 0$.
\end{corollary}
\begin{proof}
  Note that the smooth curves parametrized by $D_{k+1}$ have Clifford index at most $k-1$.
  By upper semi-continuity (\autoref{thm:sc}), we see that $[C] \not \in D_{k+1}$, and hence $[C] \not \in D$ by \autoref{prop:equality_of_divisors}.
  By \eqref{eqn:D}, this is equivalent to $K_{k-1,2}(C, \omega_C) = 0$.
\end{proof}

\section{Proof for all ribbons}\label{sec:all}
We now deduce the canonical syzygy conjecture for all ribbons from \autoref{thm:generic}.
The basic idea is to compare the Koszul cohomology group of a singular curve to that of its blow-up.
This idea appears already in the work of Voisin \cite{voi:02}, but we can exploit it much more for ribbons because they are singular everywhere!

\begin{lemma}
  \label{thm:blow-up}
  Let $C$ be a ribbon and $\beta \from C' \to C$ the blow up at a closed point.
  Then we have an inclusion $K_{k,1}(C', \omega_{C'}) \subset K_{k,1}(C, \omega_C)$.
\end{lemma}
\begin{proof}
  This follows by the same argument as in \cite[Corollary~1]{voi:02}.
  We reproduce the details for completeness.

  Since $\omega_{C'} = \beta^* \omega_C(-E)$ where $E$ is the exceptional divisor of $\beta$, we have inclusions $H^0(C', \omega^l_{C'}) \to H^0(C, \omega^l_C)$ for $l \geq 0$.
  
  Set $V'= H^0(C', \omega_{C'})$ and $V = H^0(C, \omega_C)$.
  Consider the Koszul complexes
  \[  
  \begin{tikzcd}
    \wedge^{k+1} V' \arrow{r}{\delta'} \arrow{d}{j_1} &\wedge^k V'\otimes V' \arrow{r}\arrow{d}{j_2} &\wedge^{k-1}V' \otimes H^0(C', \omega_{C'}^2) \arrow{r} \arrow{d} & \dots \\
    \wedge^{k+1} V \arrow{r}{\delta} &\wedge^k V\otimes V \arrow{r} &\wedge^{k-1}V \otimes H^0(C, \omega_{C}^2) \arrow{r}  & \dots
  \end{tikzcd}
  \]
  For any vector space $U$, the Koszul differential $\delta \from \wedge^{k+1} U \to \wedge^k U \otimes U$ has a retract $\wedge \from \wedge^k U \otimes U \to \wedge^{k+1} U$ given by the wedge product.
  Precisely, the two are related by
  \[ \wedge \circ \delta = (k+1)\id.\]
  Suppose $\alpha' \in \wedge^k V' \otimes V'$ is such that $j_2(\alpha') = \delta(\beta)$ for some $\beta \in \wedge^{k+1}V$.
  Then
  \[\wedge \circ j_2(\alpha') = \wedge \circ \delta (\beta) = (k+1)\beta.\]
  It is easy to see that $\wedge \circ j_2(\alpha') = j_1 \circ \wedge(\alpha')$.
  Set $\beta' = \wedge(\alpha')/(k+1)$.
  Then
  \[j_2 (\delta'(\beta')) = \delta(j_1(\beta')) = \delta(\beta)  = j_2(\alpha').\]
  Since $j_2$ is injective, we get $\delta'(\beta') = \alpha'$.
  In other words, any element of $\wedge^kV' \otimes V'$ that becomes a coboundary in $\wedge^kV \otimes V$ is already a coboundary.
  Therefore the map on cohomology $K_{k,1}(C',\omega_{C'}) \to K_{k,1}(C, \omega_C)$ is injective.  
\end{proof}

\begin{theorem}
  Let $C$ be a ribbon of genus $g$ and Clifford index $c$.
  Then $K_{p,2}(C, \omega_C) = 0$ if and only if $p < c$.
\end{theorem}
\begin{proof}
  The ``only if'' direction is the `easy' direction;
  it is the content of \cite[Corollary~7.3]{bay.eis:95}.

  We now prove the other direction.
  For $c = 0$, there is nothing to prove.
  Henceforth, we assume $c \geq 1$.
  In particular, $C$ is non-hyperelliptic.
  Recall the following:
  \begin{enumerate}
  \item $K_{p,q}(C, \omega_C) = 0$ implies $K_{p',q}(C, \omega_C) = 0$ for all $p' \leq p$,
  \item we have the duality $K_{p,q}(C, \omega_C) = K_{g-2-p,3-q}(C, \omega_C)^*$.
  \end{enumerate}
  Since all canonically embedded ribbons of genus $g$ and Clifford index $c$ have the same graded betti numbers (\autoref{thm:carpet}), it suffices to prove the theorem for one such ribbon.
  That is, we must show that $K_{c-1,2}(C, \omega_C) = 0$ for one ribbon of genus $g$ and Clifford index $c$.

  By the definition of the Clifford index, we have $2c \leq g-1$.
  Write $g = 2k+1-i$ and $c = k-i$ for integers $k$ and $i$ with $k \geq 1$ and $k \geq i \geq 0$.
  Let $\overline C$ be a ribbon of genus $2k+1$ and maximum Clifford index $k$.
  By \autoref{thm:blow-up-clifford}, there is a sequence of blow-ups
  \[ C_{k} \to \dots \to C_i \to \dots \to C_0 = \overline C,\]
  where $C_{k}$ is hyperelliptic.
  In this sequence, $C = C_i$ is a ribbon of genus $g = 2k+1-i$ and Clifford index $c = k-i$.
  By \autoref{thm:generic}, we know that $K_{k-1,2}\left(\overline C, \omega_{\overline C}\right) = 0$, which gives $K_{k,1}\left(\overline C, \omega_{\overline C}\right) = 0$ by duality.
  By repeated applications of \autoref{thm:blow-up}, we deduce that $K_{k,1}(C, \omega_{C}) = 0$.
  Note that $g-2-k = k-i-1 = c-1$.
  Therefore, we get $K_{c-1,2}(C, \omega_C)=0$ by duality.
  The proof of the theorem is thus complete.
\end{proof}

\bibliographystyle{abbrv}
\def\cprime{$'$}

\end{document}